\documentclass[12pt]{amsart}
\usepackage[colorlinks=true,pagebackref,hyperindex,citecolor=blue,linkcolor=red]{hyperref}
\usepackage{amsmath}
\usepackage{amsfonts}
\usepackage{setspace}
\usepackage{moreverb}
\usepackage[dvips]{graphicx}
\usepackage[latin1]{inputenc}
\usepackage[T1]{fontenc}
\usepackage{amssymb}
\usepackage{tikz}
\usepackage{color}
\usepackage[all]{xy}
\usepackage{xfrac}
\usepackage[top=1in, bottom=1in, left=1in, right=1in]{geometry}
\usepackage{mathrsfs}

\theoremstyle{definition}
\newtheorem{theorem}{Theorem}[section]

\newtheorem{corollary}[theorem]{Corollary}
\newtheorem{lemma}[theorem]{Lemma}
\newtheorem{proposition}[theorem]{Proposition}

\newtheorem{notation}[theorem]{Notation}
\newtheorem{observation}[theorem]{Observation}

\theoremstyle{definition}
\newtheorem{definition}[theorem]{Definition}

\newtheorem{example}[theorem]{Example}

\newtheorem{conjecture}[theorem]{Conjecture}

\newtheorem{remark}[theorem]{Remark}

\numberwithin{equation}{subsection}

\newcommand{\NN}{\mathbb{N}}

\newcommand{\m}{\mathfrak{m}}

\newcommand{\FDer}[1]{\stackrel{#1}{\to}}

\newcommand{\pd}{\operatorname{pd}}

\newcommand{\Shift}{\operatorname{Shift}}

\newcommand{\Reg}{\operatorname{reg}}

\newcommand{\Depth}{\operatorname{depth}}	
\newcommand{\Char}{\operatorname{char}}

\newcommand{\reg}{\operatorname{reg}}

\newcommand{\supp}{\mbox{\rm{supp}} }

\begin{document}
\newcommand{\tens}{\otimes}
\newcommand{\hhtest}[1]{\tau ( #1 )}
\renewcommand{\hom}[3]{\operatorname{Hom}_{#1} ( #2, #3 )}

\title{$b$--vectors of chordal graphs}
\author{Luis Pedro Montejano$^*$}
\address{CONACYT Research Fellow - Centro de Investigaci\'on en Matem\'aticas, Guanajuato, GTO, M\'exico}
\email{luis.montejano@cimat.mx}
\author{Luis N\'u\~nez-Betancourt$^{**}$}
\address{Centro de Investigaci\'on en Matem\'aticas, Guanajuato, GTO, M\'exico}  \email{luisnub@cimat.mx}
\thanks{$^*$ The first author was supported by the Mexican National Council on Science and Technology (C\'atedras-CONACYT)}
\thanks{$^{**}$ The second author was partially supported by NSF Grant 1502282.}
\maketitle

\begin{abstract}
The $b$--vector $(b_1,b_2\ldots,b_d)$ of a graph $G$ is defined in terms of its clique vector $(c_1,c_2\ldots,c_d)$ by the equation
$
\sum^d_{i=1}b_i(x+1)^{i-1}=\sum^d_{i=1} c_i x^{i-1},
$
where $d$ is the largest cardinality of a clique in $G$. We study  the relation of the $b$--vector of a chordal graph $G$ with some structural properties of $G$. In particular,
we show that the $b$-vector encodes different aspects of the connectivity and clique dominance of $G$.
Furthermore, we relate the $b$--vector with the Betti numbers of the Stanley-Reisner  ring associated to clique simplicial complex of $G$.
\end{abstract}

\setcounter{tocdepth}{1}
\tableofcontents

\section{Introduction}
In this manuscript we study  the $b$--vector
$(b_1,b_2\ldots,b_d)$ of a chordal graph $G$ where $d$ is the largest cardinality of a clique in $G$. We consider non-complete graphs since in this case we conclude that $b_i=1$ for every $i=1,\ldots,d$.
If $c_i$ denotes the number of cliques of $G$ with $i$ vertices, then
the clique vector is given by  $c(G)= (c_1,c_2\ldots,c_d)$.
We note that the $c$--vector is the $f$--vector of the clique complex of $G$ shifted by one.
This vector is a classical invariant  of a graph and has  been intensively studied \cite{Frohmader,FrohmaderKKThm,H08,Zykov}.

The $b$--vector of $G$ given by $\mathbf{b}(G)=(b_1,b_2\ldots,b_d)$,  is a more recent numerical invariant \cite{H08} defined by the equation
$$
\sum^d_{i=1}b_i(x+1)^{i-1}=\sum^d_{i=1} c_i x^{i-1.}
$$

Goodarzi \cite{Goodarzi} showed that the vertex connectivity of $G,$ denoted by $\kappa$, is encoded in the $b$--vector as $b_1=1$ for $i\leq \kappa$ and $b_{\kappa+1}\neq 1.$
We extend this theorem by studying the remaining $b_i$ and relate them to the number of connected components  of $G$, $W(G\setminus Y),$ after the deletion of a set $Y$, with $i-1$ vertices.

We also relate $b_i$ to clique  domination (see Definition \ref{DefDominance}). Using this notion, we define the number $d_i(G)$, which  measures how many $i$--cliques are necessary to dominate every maximal clique of order at least $i$ in $G$.
In order to establish the relation between the $b_i$ and $d_i(G)$,
we need to introduce  $\tilde{\kappa}(G)$, which  is defined as the maximum  cardinality of the intersection of any pair of maximal cliques in $G$.

 Using these combinatorial invariants of $G$, we can state the main result regarding the structure of a chordal graph encoded by its $b$-vector.

\begin{theorem}[{See Corollary \ref{b_i} and Theorem \ref{k+1}}]\label{MainThm}
Let $G$ be a a chordal graph with vertex connectivity $\kappa.$
Then,
\begin{itemize}
\item[(a)]  $b_{i}=\sum_{|Y|=i-1} (W(G-Y)-1)+1$ for $1\leq i\leq \kappa +1$, where $Y\subseteq V(G)$;
\item[(b)]  $b_{i}<\sum_{|Y|=i-1} (W(G-Y)-1)+1$ for $\kappa+2\leq i\leq d$, where $Y\subseteq V(G)$;
\item [(c)] $b_i\le d_i(G)$ for every $i=1,\ldots,d$;
\item [(d)] $b_i=d_i(G)$ for every $i>\widetilde{\kappa}(G)$;
\item [(e)] $b_i\le b_j$  for every $\widetilde{\kappa}(G)<j\leq i$.
\end{itemize}
\end{theorem}

A key component of parts of the previous theorem is the use of the Betti numbers of the Stanley-Reisner ring associated to $\Delta(G),$ the simplicial complex of cliques of $G$. This  differs from Goodarzi's approach, since he only looked at the projective dimension of this ring.
In particular, we  reinterpret parts of  Theorem \ref{MainThm} in terms of these Betti numbers using Proposition \ref{beta} and the fact that $I_{\Delta(G)}$ has a $2$-linear resolution when $G$ is chordal.
This is given by  a formula for  $\beta_{i}(R/I_{\Delta(G)})$  in terms of  the $b$--vector (see Proposition \ref{beta y b}).

\begin{theorem}
Let $G$ be a a chordal graph with $n$ vertices and vertex connectivity $\kappa$. Let $\beta_{i}(R/I_{\Delta(G)})$ be the $i$-th  Betti number of the Stanley-Reisner ring $R/I_{\Delta(G)}$ of $\Delta(G)$.
Then,
\begin{itemize}
\item[(a)] $b_{i}=\beta_{n-i}(R/I_{\Delta(G)})+1$ for every $i=1,\ldots,\kappa+1$;
\item [(b)]  $ b_i<\beta_{n-i}(R/I_{\Delta(G)})$ for every $j\ge \kappa+2$.
\end{itemize}
\end{theorem}

Another key component of the proof of Theorem \ref{MainThm} is given by algebraic shifting. In particular, we compare the exterior, $\Delta(G)^e$, and symmetric, $\Delta(G)^s$, shiftings of $\Delta(G).$
As a consequence, we prove that $I_{\Delta^e}$ and $I_{\Delta^s}$ have the same graded Betti numbers for every simplicial complex $\Delta$ such that $I_\Delta$ has a linear resolution.
This gives a partial case of a conjecture of Aramova, Herzog and Hibi \cite[Conjecture 2.3]{AHH}. Furtheremore, we also prove this conjecture  for every $t$-skeleton of $\Delta(G)$, which recover a result by Murai \cite{Murai}.

The final key component for the main result is given by an explicit description   a combinatorial shifting for chordal graphs (see Definition \ref{combishiftingraph}).

\section{Background}

\subsection{Graph terminology}

 In this manuscript we consider a simple graph $G=(V(G),E(G))$ with set of vertices $V(G)$ and  edges $E(G)$. We also assume that $G$ is not the complete graph.
We now recall some concepts from graph terminology.

The \emph{order}  of $G$   is $n=|V(G)|$.
The \emph{degree} of a vertex $v\in V(G)$, denoted as  $\deg(v)$,  is the number of vertices in $G$ adjacent to $v$ and $N(v)$ denotes the set of neighbors of $v$ in $G$. We say that $G$ is  \emph{connected} if there is a path between any two vertices of $G$.
A  \emph{vertex-cut} of  $G$ is a set of vertices whose removal disconnects $G$. Every graph that is not complete has a vertex-cut.
The \emph{vertex connectivity} $\kappa=\kappa(G)$ of a graph $G$ is the minimum  cardinality of a
vertex-cut  and a graph is $k$\emph{--connected} if $k\le \kappa(G)$.
Given a vertex-cut $Y$ of $G$, we denote as $W(G-Y)$ the number of connected components in the graph $G-Y$.

Since the $b$--vector is defined in terms of its clique vector, we recall all the concepts necessary for our study.
A \emph{clique} is a subset of vertices of a graph such that its induced subgraph is complete and a $i$--clique is a clique of order $i$.
The \emph{clique vector} $\mathbf{c}(G)$ of a graph $G$ is a vector $(c_1,c_2,\ldots,c_d)$ in $\mathbb{N}^d$, where $c_i$ is the number of cliques in $G$ with $i$ vertices and $d$ is the largest cardinality of a clique in $G$.
A \emph{maximum} clique of a graph $G$ is a clique such that there is no clique with more vertices and the \emph{clique number} is the number of vertices in a maximum clique of $G$.
A \emph{maximal clique}  in $G$ is a clique which is not contained in any other clique of $G$ and we denote the set of maximal cliques of size $i$ by  $\mathcal{C}_i(G)$  for every $1\le i\le d$.

\begin{definition}\label{DefDominance}
We say that a clique $C$ of $G$ \emph{dominates} a clique $C'$ if $C\subset C'$. A \emph{dominating $i$--clique} of $G$ is a set of cliques of order $i$ in $G$ that dominates all maximal cliques of order at least $i$ in $G$.
For every $1\le i\le d$, we take
$$
d_i(G)=\min\{ |\mathcal{D}| \;|\: \mathcal{D} \hbox{ is a dominating }i\hbox{-clique of } G    \}.
$$
We say that a dominating $i$--clique $\mathcal{D}$
is   minimum  if $|\mathcal{D}|=d_i(G)$.
\end{definition}

For a simplicial complex $\Delta$ on a set of vertices $V$ and $Y\subseteq V$, we denote as $\Delta_{|_Y}$ the simplicial subcomplex of $\Delta$ restricted in $Y$ and as $W(\Delta-Y)$ the number of connected components in $\Delta_{|_{V\setminus Y}}$.  The $t$--skeleton of $\Delta$ is given by $\Delta^{(t)}=\{\tau\in\Delta\;|\; \dim(\tau)\leq t\}$, in particular $\Delta(G)^{(1)}=G.$ The connectivity of a simplicial complex $\Delta$ is defined as the connectivity of its $1$--skeleton. The set of cliques in $G$ forms simplicial complex  $\Delta(G)$, known as the \emph{clique complex} of $G$. Then the well-known $f$\emph{-vector} of $\Delta(G)$ is exactly the clique vector of $G$.

A graph is \emph{chordal} if every cycle of length at least $4$ has a chord i.e., an edge that is not part of the cycle but connects two vertices of the cycle.
Given the clique vector $\mathbf{c}(G)$ of a chordal graph $G$, we define the vector $b$--vector of $G$, $\mathbf{b}(G)=(b_1,b_2,\ldots,b_d)$, defined as
\begin{equation}\label{eq1}
\sum\limits_{i=1}^{d}b_i(x+1)^{i-1}=\sum\limits_{i=1}^{d}c_ix^{i-1}.
\end{equation}

We seek to show that as $\mathbf{c}(G)$ give us the number the $i$--cliques of $G$,   $\mathbf{b}(G)$  also give us  structural information if $G$ is chordal.

\subsection{Free resolutions and Stanley-Reisner rings}
In this subsection we consider a polynomial ring $R=K[x_1,\ldots,x_n]$ as a $\NN^n$--graded ring with $\deg(x_i)=e_i$
where $e_i$ is the vector with one in the  $i$--th entry and zeros elsewhere.
We take $\m=(x_1,\ldots,x_n).$
Let $M$ be a graded finitely generated $R$--module.
Then, there exists a resolution by graded free modules
$$
0\to  F_p \FDer{\varphi_{i-1}} F_{i-1}\to \ldots \FDer{\varphi_0} F_o\to M\to 0,
$$
where $\varphi_i$ is represented by a matriz with homogebeous entries in $\m.$
We can  write $F_i = \bigoplus_j R(-j)^{\beta_{i,\alpha} (M)}$, where $R(-\alpha)$ denotes a rank one free module with an  homogeneous generator in degree $-\alpha$.
By The Hilbert Syzygy Theorem, $p\leq n$.
The projective dimension of $M$ is defined by $\pd(M)=p.$ By the Auslander-Buchsbaum formula, we have that
$\Depth(M)=n-\pd(M)$.
The numbers $\beta_{i,j}(M)$ are important invariants for $M$, with a vast number of applications in algebra, geometry, and topology.
The Castelnuovo-Mumford regularity is defined by $\reg(M)=\max\{j\;|\; \beta_{i,\alpha}(M) \neq 0\hbox{ with }|\alpha|= i+j\}$.

We say that $M$ has \emph{$t$--linear resolution} if every homogeneous minimal generator of $M$ has degree $t$ and $\beta_{i.i+j}(M)=0$ for $j\neq t$.

We refer Peeva's book on free resolutions \cite{PeevaGrSyz} and Eisenbud's book on syzygies \cite{EisenbudSyzygy}
for more details about free resolutions and their geometry.

An ideal generated by monomials is square free, if it is radical.
There is a well-known  bijection between the squarefree monomial ideals in $R$ and  simplicial complexes in $n$.

Given a simplicial complex $\Delta$, the square free monomial associated to $\Delta$ is defined by
$$I_{\Delta} = \left( x^\sigma\; |\;\sigma  \notin \Delta \right).$$
The quotient $R / I_{\Delta}$ is called the Stanley-Reisner ring of $\Delta$.

Given a square-free monomial ideal $I$,  the simplicial complex associated to $I$ is defined by
$$\Delta = \left\lbrace \sigma \subseteq [n] \; | \; x^\sigma \notin I \right\rbrace.$$

We refer to the book by Miller and Sturmfeels  on combinaturial commutative algebra \cite{MillerSturmfels} and the survey by Francisco, Mermin and Schweig on Stanley-Reisner theory \cite{SurveySR}  for more details about these rings.

Hochster's Formula \cite[Theorem 5.1]{HochsterFormula} is an important result that connects the topology of simplicial comlexes to free resolutions, which plays a key role in this manuscript.
Let $W=\supp(\alpha).$
Then,
$$\beta_{i+1,\alpha}(R/I_\Delta)=\beta_{i,\alpha}(I_\Delta)= \tilde{H}^{|W|-i-2}(\Delta_{|_W }; K)$$
for every $i\in\NN$ and $\alpha\in\NN^n.$

\section{Threshold Graphs}
In this section, we focus on the $b$--vector of threshold graphs. This is a key part of our strategy since the $b$--vector determines a unique threshold graph, and therefore, its graph properties. To expand our result to chordal graphs we use results from shifting theory in Subsection \ref{SubsectionShift}.

\subsection{Definition and basic properties}
A threshold graph is a graph that can be constructed from a one-vertex graph by the following two operations: addition of one isolated vertex, denoted by $D$,  and addition of a vertex connected to all other vertices, denoted by $S$. We set the convention of always start with an $S$--operator. Thus, there is a bijection between threshold graphs and words on the alphabet $\{S,D\}$ beginning with an $S$ reading from left to right.


\begin{figure}[htb]\label{FigExThreshold}
\begin{center}
 \includegraphics[width=.2\textwidth]{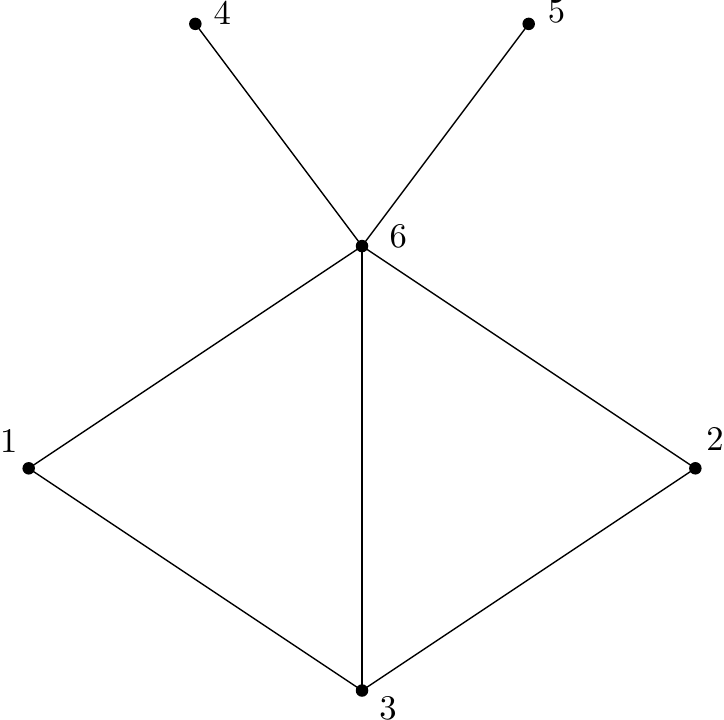}
\caption{Graph corresponding to the word SDSDDS }\label{threshold}
\end{center}
\end{figure}

We consider a vertical line before every $S$ in the word corresponding to the threshold graph $T$, thus breaking the word into subwords.
Goodarzi \cite{Goodarzi} showed that the entry
 $b_i$ of the $b$--vector  $\mathbf{b}(G)$ is the length of the $(d-i+1)$--th subword for every $i=1,\ldots,d$. As a consequence, there is a bijection between $b$--vectors and threshold graphs.
 For instance, the graph in Figure \ref{FigExThreshold}, which has the word $SDSDDS$,
 breaks to $|SD|SDD|S$ and $(b_1,b_2,b_3)=(1,3,2)$.


Given a threshold graph $T$, we recall that $\mathcal{C}_i(T)$ is the set of maximal cliques of size $i$ in $G$ and we denote as  $C_i^s$ the $i$--clique of $T$  conformed by all vertices corresponding to all the letters $S$ in the  $(d-j+1)$--th subwords, $1\le j\le i$.
Many properties of a threshold graph can be read off from its word. We collect some simple observations.

\begin{observation}\label{obs}
Let $T$ be a threshold graph with clique number $d$, then the following hold:
\begin{itemize}
  \item [(a)] The number of times that $S$ appears in $T$ is the clique number of $T$.
  \item [(b)] The connectivity of $T$ is the number of consecutive $S$ appearing at the end of the word $T$.
  \item [(c)]   There is only one minimum vertex-cut in $T$.
  \item [(d)] For any vertex-cut $Y$ of $T$, the graph $T-Y$ has at most one component with at least $2$ vertices.
  \item [(e)] $\mathcal{C}_i(T)\cup C_i^s$ is the unique  minimum  dominating $i$--clique of $T$ for every $1\le i\le d$.
\end{itemize}
\end{observation}

We point out that there are others characterization of threshold graphs. We recall one that is useful in our study.
\begin{theorem}[\cite{MahadevPeled}]\label{Mahadev}
A graph is threshold if and only if it does not contains the graphs $G_1$, $G_2$ or $G_3$ of Figure \ref{G_1G_2G_3}
as an induced subgraph.
\end{theorem}
\begin{figure}[htb]
\begin{center}
 \includegraphics[width=.4\textwidth]{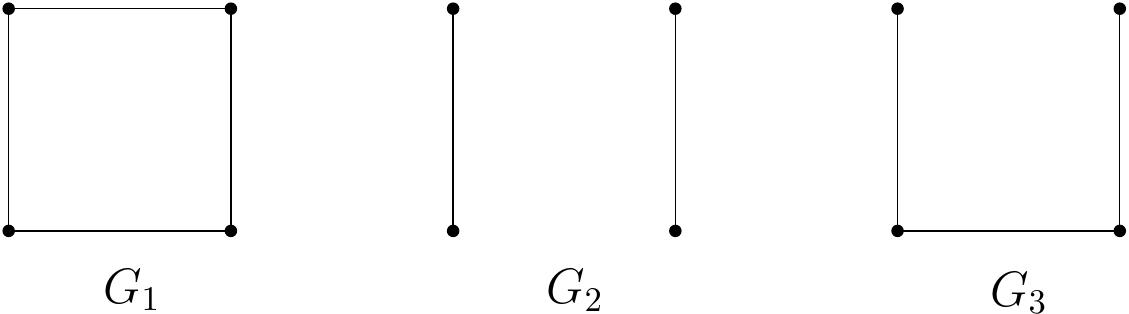}
\caption{ }\label{G_1G_2G_3}
\end{center}
\end{figure}

\subsection{$b$--vectors of threshold graphs}

We start this section with a result that establishes a strong relation between the $b$--vector and the graphs structure.

\begin{proposition}\label{dominate}
Let $T$ be a threshold graph  with  vertex connectivity $\kappa$. 
 Then,
\begin{itemize}
\item [(a)]  $b_1=b_2=\ldots=b_{\kappa}=1$;
 \item [(b)] $b_i-1=|\mathcal{C}_i(T)|$ for every $\kappa+1\le i\le d-1$ and $b_d=|\mathcal{C}_d(T)|$;
 \item [(c)] $b_i=d_i(T)$ for every $i=1,\ldots,d$;
 \item [(d)] $b_{\kappa+1}=W(T-Y)$ for a minimum vertex-cut $Y$.
 \end{itemize}
\end{proposition}
\begin{proof}$ $
\begin{itemize}
\item[(a)] It follows from Observation \ref{obs}(b).

\item[(b)]
Let $\kappa+1\le i\le d-1$.
Then, each maximal clique of size $i$ in $T$ is forming by a vertex in $T$ corresponding to a letter $D$ in the $(d-i+1)$--th subword together with the vertices in $T$  corresponding to all the letters $S$ in the  $(d-j+1)$--th subwords for every $1\le j\le i-1$.
Since the $(d-i+1)$--th subword has exactly one $S$ and $b_i$ corresponds to the cardinality of letters of its $(d-i+1)$--th subword, we obtain that $b_i-1=|\mathcal{C}_i(T)|$.
Finally, if we put the letter $D$ instead $S$ in the first position of the word corresponding to $T$, we  obtain the same graph $T$. Hence, the previous argument holds for $i=d$ obtaining that $b_d=|\mathcal{C}_d(T)|$.

\item[(c)] We note that $\mathcal{C}_i(T)\cup C_i^s$ is the unique  minimum  dominating $i$--clique of $T$ by Observation \ref{obs}(e). If $1\le i\le \kappa$, then the set $\mathcal{C}_i(T)=\emptyset$ and the clique $C_i^s$ is a minimum  dominating $i$--clique of $T$ proving that $d_i(T)=1$ for every $1\le i\le \kappa$. We conclude that $b_i=d_i(T)=1$ for every $1\le i\le \kappa$  by (a).
If $\kappa+1\le i\le d$, we have that $b_i-1=|\mathcal{C}_i(T)|$ for every $\kappa+1\le i\le d-1$ and $b_d=|\mathcal{C}_d(T)|$ by (b).
Since  $\mathcal{C}_i(T)\cup C_i^s$ is a minimum  dominating $i$--clique in $T$  by Observation \ref{obs}(e), it follows that $b_i=d_i(T)$ for every $\kappa+1\le i\le d$, because  $i=d$ implies $C_i^s\in \mathcal{C}_i(T)$.

\item[(d)] By Observation \ref{obs}(b)(c), there is only one minimum vertex-cut $Y$ in $T$, which is formed by the vertices of $T$ corresponding to all consecutive $S$ appearing at the end of the word corresponding to $T$. Therefore,  $b_{\kappa+1}$ is the number of components in $T-Y$.
\end{itemize}
\end{proof}

One of the main objectives of this paper is to extend the previous proposition to chordal graphs and  to compare the $b$--vector of a chordal graph $G$ with the Betti numbers of $\Delta(G)$.
We end this subsection with a result that further the relation of the  $b$--vector with connectivity.


\begin{proposition}\label{k+2}
Let $T$ be a threshold graph with vertex connectivity $\kappa$. 
  Then,
$$
    b_{i+1}<\sum\limits_{\substack{|Y|=i}} (W(T-Y)-1)
  $$ for every $\kappa+1\le i\le d-1$ where $Y\subseteq V(T)$.
\end{proposition}
\begin{proof}
We observe that if $Y\subseteq V(T)$ is not a vertex-cut of $T$, then $W(T-Y)-1=0$. Hence the sum $\sum\limits_{\substack{|Y|=i}} (W(T-Y)-1)$ only consider vertex-cuts $Y$ of $T$.
 Let denote as $B_i$ the set of vertices in $T$ that correspond to the letters in the $(d-i+1)$--th subword  for every $i=1,\ldots,d$ (i.e., $|B_i|=b_i$). By Proposition \ref{dominate}(a), for every $i=1,\ldots,\kappa$ the set  $B_i$ has only one vertex, say $x_i$,  and by Observation \ref{obs}(c) the only  minimum vertex-cut in $T$  is the set $\{x_i\}_{i=1}^{\kappa}$.

   Let $V_{\kappa+\ell}$ be the number of all vertex-cuts of $T$ with cardinality $\kappa+\ell$, then clearly $V_{\kappa+\ell}\le \sum\limits_{\substack{|Y|=\kappa+\ell}} (W(T-Y)-1)$ since $W(T-Y)-1\ge 1$ for every vertex-cut $Y$ of $T$ with cardinality $\kappa+\ell$. Therefore, it is enough to show that $b_{\kappa+1+\ell}< V_{\kappa+\ell}$ for every $1\le\ell\le d-\kappa-1$.

   Suppose that $T$ has order $n$, let $a\in B_{\kappa+1}$ be such that its correspond letter is $D$ and consider the set $\bigcup\limits_{j=\kappa+1}^{d} B_j-\{a\}$ of cardinality $n-\kappa-1$. As $d\le n-1$ since $b_{\kappa+1}\ge 2$, we have that $\ell<n-\kappa-1$. Then, for each $L\subset\bigcup\limits_{j=\kappa+1}^{d} B_j-\{a\}$ with cardinality $|L|=\ell$ we have that $Y=\{x_j\}_{j=1}^{\kappa}\cup\{L\}$ is a vertex-cut of $T$ with cardinality $\kappa+\ell$ since vertex $a$ is isolate in $T-Y$. Therefore, there are at least ${n-\kappa-1\choose \ell}$ vertex-cuts of $T$ with cardinality $\kappa+\ell$ concluding that ${n-\kappa-1\choose \ell}\le V_{\kappa+\ell}$. On the other hand, we notice that $b_{\kappa+1+\ell}<n-\kappa-1$ for every   $1\le \ell\le d-\kappa-1$, otherwise we would have that $b_j=1$ for every $j=1,\ldots,\kappa+1$, a contradiction since $b_{\kappa+1}\ge2$. Thus, as $n-\kappa+1\le {n-\kappa-1\choose \ell}$ for every $1\le \ell< n-\kappa-1$,  we conclude that $b_{\kappa+1+\ell}<V_{\kappa+\ell}$.
\end{proof}

\section{Algebraic tools}

\subsection{Betti numbers and connectivity}
 The following result has already been pointed out by several authors \cite[Theorem 6]{Goodarzi}
 \cite[Corollary 1.2]{KatzmanBettiGraph}. We include a proof in the contexts we need for the sake of completeness.
\begin{proposition}\label{beta}
Let $G$ be a graph of order $n$ and let $\Delta=\Delta(G)$.
Then,
$$
 \beta_{i,i+1}(R/I_\Delta)=\sum_{|Y|=n-i-1} (W(G-Y)-1),
$$  for every $1\le i\le n$, where $Y\subseteq V(G)$.
  As a consequence,
$G$ is $k$--connected if and only if
$\beta_{i,i+1}(R/I_\Delta) =0$ for all $i\geq n-k.$ In particular,
$$
\kappa(G) =\max\{ k\; | \; \beta_{i,i+1}(R/I_\Delta) =0\hbox{ for all }i\geq n-k\}.
$$
\end{proposition}
\begin{proof}
From Hochster's Formula \cite[Theorem 5.1]{HochsterFormula}, we have that
\begin{align*}
\beta_{i,i+1}(R/I_\Delta)=\beta_{i-1,i+1}(I_\Delta)&=\sum_{|Y|=i+1}
\dim_K \widetilde{H}^0 (\Delta_{|_Y};K)\\
&=\sum_{|Y|=i+1}  (W(G_{|_Y})-1)\\
&=\sum_{|Y|=n-i-1} (W(G- Y)-1)\\
&=\sum_{|Y|=n-i-1}  (W(G- Y)-1).
\end{align*}
The rest follows from observing that $\beta_{i,i+1}(Y/I_\Delta)=0$ if and only if there is no vertex-cut $Y$ such that $|Y|=n-i-1.$
\end{proof}

The previous result implies the following result, which is already known  \cite{KatzmanBettiGraph,Goodarzi}.

\begin{corollary}[{\cite[Theorem 6]{Goodarzi}}]\label{CorDepth}
Let $G$ be a graph and set $\Delta=\Delta(G).$
Then, $$\Depth(R/I_\Delta)\leq \kappa(G)+1.$$
Furthermore, the equality is reached if $G$ is chordal.
\end{corollary}


\subsection{Shifting theory}\label{SubsectionShift}
In order to continue our study of $b$--vectors,  we need to recall a few terms and properties of shifting theory for simplicial complex.

\begin{definition}
A simplicial complex $\Delta$ on $[n]$ is shifted if, for $\sigma\in\Delta, i\in \sigma,$ and
$j\in [n]$ with
$j>i,$ one has $(\sigma\setminus \{ i\})\bigcup\{ j\}\in\Delta$.
\end{definition}

\begin{definition}
A shifting operation on $[n]$ is a map which associates each simplicial
complex $\Delta$ on $[n]$ with a simplicial complex $\Shift(\Delta)$ on $[n]$ and which satisfies the following conditions:
\begin{itemize}
\item[(a)]  $\Shift(\Delta)$ is shifted;
\item[(b)] $\Shift(\Delta)=\Delta$ if $\Delta$ is shifted;
\item[(c)] $f(\Delta) = f(\Shift(\Delta))$;
\item[(d)] $\Shift(\widetilde{\Delta})\subseteq\Shift(\Delta)$ if $\widetilde{\Delta}\subseteq\Delta$ and $\widetilde{\Delta}^{(0)}=\Delta^{(0)}$.
\end{itemize}
\end{definition}

\begin{definition}
We say that a shifting operation $\Shift(-)$ is compatible with Alexander Duality if
$\Shift(\Delta^\vee)=\Shift(\Delta)^\vee$
\end{definition}

There are several well-known shifting operations. Among these operations, the exterior shifting $\Delta^e$ and the symmetric shifting $\Delta^s$ stand up.
We point out that $\Delta^s$ is only defined in characteristic zero.
 We omit the definitions as its not needed for our purposes. However, we refer the interested reader in the book by Herzog and Hibi \cite[Chapter 11]{BookMonomial}.

 We recall a few properties for the exterior and symmetric shifting.

\begin{theorem}[{\cite[Theorem 11.4.1]{BookMonomial}}]\label{ThmShiftProp}
Let $\Delta$ be any simplicial complex.
The following statements holds:
\begin{enumerate}
\item $\Depth(R/I_\Delta)=\Depth(R/I_{\Delta^e})$.
\item $(\Delta^e)^\vee=(\Delta^\vee)^e$
\item $\Reg(R/I_\Delta)=\Reg(R/I_{\Delta^e})$.
\end{enumerate}
If the ground field has characteristic zero, then the same statements hold for the symmetric shifting.
\end{theorem}

We now mention a conjecture that relates the Betti numbers of algebraic and symmetric shifting.

\begin{conjecture}[{\cite[Conjecture 2.3]{AHH}}]\label{Conj}
Let $\Delta$ be a simplicial complex and suppose that $\Char(K)=0.$
Then,
$$\beta_{i,j}(I_{\Delta^s})\leq \beta_{i,j }(I_{\Delta^e}).$$
\end{conjecture}

In order to study the previous conjecture and the  $b$-vector of a chordal graph, we need to recall the definition and properties of the $h$-vector of a simplicial complex.

\begin{definition}
Let $\Delta$ be a simplicial complex of dimension $d-1$ and
let $f(\Delta)=(f_{-1},\ldots,f_{d-1})$ be the $f$-vector of $\Delta$.
The $h$-vector of $\Delta,$ $h(\Delta)=h_0,\ldots,h_d)$, is defined by
$$
\sum^d_{i=0}h_it^i=\sum^d_{i=0} f_{i-1} t^i (1-t)^{d-i}.
$$
\end{definition}

\begin{remark}\label{h-vector f-vector}
We have that
$$
h_j=\sum^j_{i=0}(-1)^{j-i}\binom{d-i}{j-i} f_{i-1} \quad \& \quad f_{j-1}=\sum^j_{i=0} \binom{d-i}{j-i} h_i.
$$
\end{remark}

We now recall a result  that would imply that the $f$-vector of certain simplicial complexes determine its Betti numbers.

\begin{theorem}[{\cite[Corollary 3.4]{H10}}]\label{ThmBettiH}
If $I_\Delta$ has a $t$-linear free resolution,
then
$$\beta_i(R/I_\Delta)=\sum^{t+i}_{\ell=0} (-1)^{\ell+i+1} h_{t+i-\ell} \binom{n-d}{\ell}$$
\end{theorem}

We  are able to answer Conjecture \ref{Conj} for ideals with a linear free resolution.

\begin{proposition}\label{PropPreservationBetta}
Let $\Delta$ be a simplical complex such that $I_\Delta$ has a $t$-linear free resolution. Then,
$\beta_{i,j}(R/I_\Delta)=\beta_{i,j}(R/I_{\Delta^e})$.
Furthermore, If $\Char(K)=0,$ then
$$
\beta_{i,j}(R/I_{\Delta^s})=\beta_{i,j}(R/I_{\Delta})=\beta_{i,j}(R/I_{\Delta^e}).
$$
\end{proposition}
\begin{proof}
Since $I_\Delta$ has a free $t$-linear reolution,
$I_{\Delta^s}$ and $I_{\Delta^e}$ also have a free resolution by Theorem \ref{ThmShiftProp}.
We conclude that
$$
\beta_{i,j}(R/I_{\Delta^s})=\beta_{i,j}(R/I_{\Delta})=\beta_{i,j}(R/I_{\Delta^e})=0.
$$
for $j\neq i+t-1$.
In addition,
$\beta_{i}(R/I_{\Delta^s})=\beta_{i,j}(R/I_{\Delta^s})$,
$\beta_{i,j}(R/I_{\Delta})=\beta_{i,j}(R/I_{\Delta})$, and
$\beta_{i,j}(R/I_{\Delta^e})=\beta_{i,j}(R/I_{\Delta^e}).$
Since shifting preserves $f$-vector, it also preserves $h$-vectors. Then, exterior and symmetric shifting preserve total Betti numbers in this case by Theorem \ref{ThmBettiH}.
Hence,
$$
\beta_{i,j}(R/I_{\Delta^s})=\beta_{i,j}(R/I_{\Delta})=\beta_{i,j}(R/I_{\Delta^e}).
$$
for $j=i+t-1.$
\end{proof}

As a consequence of Theorem \ref{ThmBettiH} and Proposition \ref{beta}, we prove that shifting preserves  numbers related to connectivity for chordal graphs.

\begin{corollary}\label{CorConnectivity}
 Let  $G$ be a chordal graph with vertex connectivity $\kappa$ and set $\Delta=\Delta(G).$
Then,
$$
\sum_{|Y|=\kappa} (W(\Delta- Y)-1)=
\sum_{|Y|=\kappa} (W(\Delta^e- Y)-1),
$$ where $Y\subseteq V(G)$.
Furthermore, if $\Char(K)=0$ ,
then
$$
\sum_{|Y|=\kappa} (W(\Delta- Y)-1)=
\sum_{|Y|=\kappa} (W(\Delta^s- Y)-1).
$$
\end{corollary}
\begin{proof}
This follows immediately from Propositions \ref{beta} and \ref{PropPreservationBetta}.
\end{proof}

We now recall an structural result that characterizes clique complexes of threshold graphs.

\begin{proposition}[{\cite[Theorem 2]{Carly}}]\label{TeoCarly}
The graph $G$ is threshold if and only if  $\Delta(G)$ is shifted.
\end{proposition}

Thanks to the previous result, we are able to show that the exterior and symmetric shifting yield the same complex for clique complexes of chordal graphs.

\begin{theorem}\label{TeoSameShift}
Let $G$ be a chordal graph. Suppose that $\Char(K)=0$.
Then $\Delta(G)^e=\Delta(G)^s.$
\end{theorem}
\begin{proof}
Since $G$ is a chordal graph, $I_{\Delta(G)}$ has a $2$--linear resolution.
Then,  $I_{(\Delta(G))^e}$ and $I_{(\Delta(G))^s}$ have a $2$--linear resolution.
Then,   $I_{(\Delta(G))^e}$ and $I_{(\Delta(G))^s}$ are the monomial edge ideals of certain graph.
As a consequence,
 $(\Delta(G))^e$ y $(\Delta(G))^s$ are clique complexes of certain graphs, say $G_1$ and $G_2$.
 By Theorem \ref{TeoCarly}, $G_1$ and $G_2$ are threshold graphs.
We obtain that
$$
f(\Delta(G_1))=f((\Delta(G))^e)=
f(\Delta)
=
f((\Delta(G))^e)=f(\Delta(G_2)).
$$
Then the clique vector of $G_1$ and $G_2$ are equal, concluding that $G_1$ and $G_2$ have the same $b$-vector, so $G_1=G_2$.
Hence
$\Delta(G_1)=\Shift_1(\Delta(G))=\Shift_2(\Delta(G))=\Delta(G_2).$
\end{proof}

As a consequence of Theorem \ref{TeoSameShift}, we show  Conjecture \ref{Conj} for clique complex of chordal graphs. We point out that the fact that $G^e=G^s$ was previously know \cite{Murai}.

\begin{proposition}\label{Prop e=s}
Suppose that $\Char(K)=0$. Then $\Delta(G)^e=\Delta(G)^s$ for every chordal graph.
As a consequence,
$(\Delta(G)^{(i)} )^e=(\Delta(G)^{(i)})^s$ and Conjecture \ref{Conj} holds for
$\Delta(G)^{(i)}$
for every $i\in\NN$.
\end{proposition}
\begin{proof}
This is an immediate consequence from Theorem \ref{TeoSameShift}, and the fact that $\Shift(\Delta^{(i)})=(\Shift(\Delta))^{(i)}$ for every shifting operation $\Shift(-)$ and $i\in\NN.$
\end{proof}

\section{Chordal graphs}

In this section we prove our main results regarding chordal graphs. We start by recalling the definition and properties of a perfect elimination order.

A vertex $v$ of a graph $G$ is called \emph{simplicial} if its adjacency set $N(v)$ induces a clique. An \emph{elimination ordering} $\sigma$ of a graph $G$ is a bijection $\sigma:\{1,\ldots,n\}\to V(G)$. Thus, $\sigma(i)$ is the $i$th vertex in the elimination ordering and for $v\in V(G)$, $\sigma^{-1}(v)$ gives the position of $v$ in $\sigma$. A \emph{perfect elimination ordering} (PEO) is an
elimination ordering  $\sigma=(v_1,v_2,\ldots,v_n)$ where $v_i$ is a simplicial vertex in the
subgraph induced by $\{v_i,v_{i+1},\ldots,v_n\}$, $1\le i\le n$. Given a PEO $\sigma$ of a graph $G$, the \emph{monotone adjacency set} of $v$, denoted $N_{\sigma}(v)$, is given by
$$N_{\sigma}(v):=\{u\in N(v) \mbox{ }\vert \mbox{ } \sigma^{-1}(v)<\sigma^{-1}(u)\}$$ and denote $n_{\sigma}(v):=|N_{\sigma}(v)|$. For a
 maximal clique $C$ of $G$, let define  $$s(C):=\{x\in C \mbox{ }\vert \mbox{ } N_{\sigma}(x)\not\subseteq C\}.$$ For any set of vertices $U=\{u_{1},u_{2},\ldots,u_r\}$ of $G$, we say that $U$ is \emph{well ordered} if $\sigma^{-1}(u_{i})<\sigma^{-1}(u_{j})$ for every  $i,j\in \{1,\ldots,r\}$ with $i<j$.

As shown below in Theorem \ref{Dirac}, every nontrivial chordal graph has at least two
simplicial vertices.

\begin{theorem}[{\cite{D61}}]\label{Dirac}
Every chordal graph has a simplicial vertex. If $G$ is not
the complete graph, then it has two nonadjacent simplicial vertices.
\end{theorem}

We know recall a characterization of chordal graphs related to PEO.

\begin{theorem}[{\cite{FulkersonGross}}]\label{Fulkerson}
A graph is chordal if and only if it has a PEO.
\end{theorem}

The next result shows the existence of a special PEO satisfying properties that are useful for the rest of the paper.

\begin{lemma}\label{K_d}
Let $G$ be a chordal graph of order $n$, let $K=\{x_1,x_2\ldots,x_k\}$  be a maximal clique of $G$ and let $C$ be any maximal clique of $G$. Then there exists a PEO  $\sigma$ of $G$ such that:
\begin{itemize}
  \item [(a)]  $\sigma^{-1}(u)<\sigma^{-1}(x)$ for every $x\in K$ and every $u\in V(G)-K$, moreover $\sigma^{-1}(x_i)=n-i+1$ for every $i=1,\ldots,k$;
  \item [(b)]  $\sigma^{-1}(u)<\sigma^{-1}(v)$ for every $u\in C-s(C)$ and every $v\in s(C)$;
  \item [(c)] if $|s(C)|<i\le |C|$, then there exists a unique vertex $u\in C-s(C)$ such that $n_{\sigma}(u)=i-1$;
   \item [(d)] if $C'$ is another maximal clique of $G$ such that $C'\cap C\neq\emptyset$, then $\sigma^{-1}(u)<\sigma^{-1}(v)$ for every $u\in C\setminus C'$ and every $v\in C\cap C'$ or   $\sigma^{-1}(u)<\sigma^{-1}(v)$ for every $u\in C'\setminus C$ and every $v\in C\cap C'$.
\end{itemize}

\end{lemma}
\begin{proof}
We construct a  PEO  $\sigma$ of $G$ as follows.  Let $\mathcal{C}$ be the set of maximal cliques of $G$ and suppose that $|\mathcal{C}|=r$.  There exists two nonadjacent simplicial vertices in $G$ by Theorem \ref{Dirac}. Then, one of them is a vertex $u_1\in V(G)-K$. Suppose that $u_1\in C_1$ for some $C_1\in \mathcal{C}\setminus\{K\}$. Set $P_1=\{u\in C_1 \mbox{ }\vert \mbox{ } u \mbox{ is simplicial in } G\}$.
We note that the graph $G_2=G-P_1$ is chordal and  for every $C\in \mathcal{C}-{C_1}$, $C$ is a maximal clique in $G_2$. By Theorem \ref{Dirac}, there exists a vertex $u_2\in V(G_2)-K$ simplicial in $G_2$. Suppose that $u_2\in C_2$ for some $C_2\in \mathcal{C}\setminus\{K,C_1\}$.  Let $P_2=\{u\in C_2 \mbox{ }\vert \mbox{ } u \mbox{ is simplicial in } G_2\}$. Proceeding in this way for every $i=3,4,\ldots,r-1$, we obtain that
the graph  $G_i=G_{i-1}-P_{i-1}$ is chordal. In addition,  $C$ is a maximal clique in $G_i$ for every $C\in \mathcal{C}\setminus\{C_1,C_2,\ldots,C_{i-1}\}$.
 Then, by Theorem \ref{Dirac}, there exists a vertex $u_i\in G_i-K$ simplicial in $G_i$. Suppose that $u_i\in C_i$ for some $C_i\in \mathcal{C}-\{K,C_1,C_2,\ldots,C_{i-1}\}$. Set $P_i=\{u\in C_i \mbox{ }\vert \mbox{ } u \mbox{ is simplicial in } G_i\}$. We define $\sigma$ as follows. For every $i,j=1,\ldots,r-1$ with $i<j$, every $p_i\in P_i$ and every $p_j\in P_j$, we take  $\sigma^{-1}(p_i)<\sigma^{-1}(p_j)$ and  $\sigma^{-1}(x_i)=n-i+1$ for every $i=1,\ldots,k$. Then, $\sigma$ is a PEO of $G$ by construction.
We now show that $\sigma$ satisfies the condition (a)-(d).
\begin{itemize}
  \item [(a)]  It follows by the construction of $\sigma$.

  \item [(b)] We notice that $\mathcal{C}=\{C_1,C_2\ldots,C_r\}$ where $C_r=K$. By construction of $\sigma$, the sets $S_i$ and $s(C_i)$ are a partition of $C_i$  for every $i=1,\ldots,r$. 
  Suppose that  $C=C_i$ for some $i=1,\ldots,r$. If $i=r$, then $s(C)=\emptyset$. Then, we can assume that $i=1,\ldots,r-1$. Let $v\in s(C)$. By definition of the set $s(C)$, $N_{\sigma}(v)\not\subseteq C$. Therefore, there exists a vertex $z\in N_{\sigma}(v)$ such that $z\not\in C$. Since $\sigma^{-1}(v)<\sigma^{-1}(z)$ and $z\not\in C$, it follows that $z\in \bigcup\limits_{j=i+1}^{r}C_j$. Then, by construction of $\sigma$, we have that $\sigma^{-1}(u)<\sigma^{-1}(z)$ for every $u\in S_i=C\setminus s(C)$.
  
We proceed  by contradiction.   Suppose that there exits a vertex $u\in C\setminus s(C)$ with $\sigma^{-1}(v)<\sigma^{-1}(u)$. Then, $u\in N_{\sigma}(v)$. Since $N_{\sigma}(v)$ is a clique, we have that $uz\in E(G)$. As $z\not\in C$ and $N_{\sigma}(u)\subseteq C$, we have that $z\not\in N_{\sigma}(u)$. Since $uz\in E(G)$, we conclude that $u\in N_{\sigma}(z)$.
This means that $\sigma^{-1}(z)<\sigma^{-1}(u)$, a contradiction.

  \item [(c)] Suppose that $C=\{u_1,u_2,\ldots,u_{|C|}\}$ is well ordered. By (b) it follows that $u_j\not\in s(C)$ for every $j=1,\ldots,|C|-|s(C)|$. Since $N_{\sigma}(u_j)\subset C$ for every $j=1,\ldots,|C|-|s(C)|$, we have that $n_{\sigma}(u_j)=|C|-j$. Therefore, there exists a unique vertex $u\in C \setminus s(C)$ such that $n_{\sigma}=i-1$, because   $|s(C)|<i\le |C|$.

   \item [(d)] Let $v\in C\cap C'$ . We consider the following cases.

\noindent\underline{\emph{ Case (i):  $N_{\sigma}(v)\cap (C'\setminus C)\neq\emptyset$.}}
  In this case $N_{\sigma}(v)\cap (C\setminus C')=\emptyset$, because $N_{\sigma}(v)$ is a clique. 
  We deduce that $\sigma^{-1}(u)<\sigma^{-1}(v)$ for every $u\in C\setminus C'$, since $u\not\in N_{\sigma}(v)$. Let $v'\in C\cap C'$ be any other vertex different from $v$. If $v'\in N_{\sigma}(v)$, then $\sigma^{-1}(u)<\sigma^{-1}(v)<\sigma^{-1}(v')$  for every $u\in C\setminus C'$. If $v'\not\in N_{\sigma}(v)$. As $N_{\sigma}(v)\subset N_{\sigma}(v')$, it follows that $N_{\sigma}(v')\cap (C'\setminus C)\neq\emptyset$. Then, by a similar argument applied to $v'$, we conclude that
 $\sigma^{-1}(u)<\sigma^{-1}(v')$ for every $u\in C\setminus C'$.

\noindent  \underline{\emph{Case (ii): $N_{\sigma}(v)\cap (C\setminus C')\neq\emptyset$.}} 
The proof is analogous to Case (i).

\noindent\underline{\emph{Case (iii): $N_{\sigma}(v)\subset (C\cap C')$.}}
 In this case, $\sigma^{-1}(u)<\sigma^{-1}(v)$ for every $u\in C\cup C'\setminus (C\cap C')$.  Let $v'\in C\cap C'$ be any other vertex different from $v$. If $v'\in N_{\sigma}(v)$, then $\sigma^{-1}(u)<\sigma^{-1}(v)<\sigma^{-1}(v')$ for every $u\in C\cup C'\setminus (C\cap C')$. If $v'\not\in N_{\sigma}(v)$, we assume  that $N_{\sigma}(v')\cap (C'\setminus C)\neq\emptyset$. Then by a similar argument to Case (i), we have that $\sigma^{-1}(u)<\sigma^{-1}(v')$ for every $u\in C\setminus C'$. The case when $N_{\sigma}(v')\cap (C\setminus C')\neq\emptyset$ is analogous. The case $N_{\sigma}(v')\subset C\cap C'$ implies that $\sigma^{-1}(u)<\sigma^{-1}(v')$ for every $u\in C\cup C'\setminus(C\cap C')$.

We conclude, in any case, that $\sigma^{-1}(u)<\sigma^{-1}(v)$ for every $u\in C\setminus C'$ and every $v\in C\cap C'$ or   $\sigma^{-1}(u)<\sigma^{-1}(v)$ for every $u\in C'\setminus C$ and every $v\in C\cap C'$.
\end{itemize}
\end{proof}


\subsection{A combinatorial shifting of chordal graphs}
In this subsection, we define a combinatorial  procedure which takes chordal graphs to threshold graphs preserving the clique complex.
This procedure results  a shifting operation which allows us  to have a combinatorial description of shifting in order to use in our results
(for a different  study of combinatorial shifting of chordal graphs we refer to Murai's work on shifting for graphs \cite{Murai}).

\begin{definition}\label{combishiftingraph}
Let $G$ be a chordal graph with order $n$ and let $K_d=\{x_1,x_2\ldots,x_d\}$ be a maximum clique of $G$.
Let  $\sigma$ be a PEO  of $G$ that satisfies the conditions of Lemma \ref{K_d} for $G$ and $K_d$ (in particular $\sigma^{-1}(x_i)=n-i+1$ for every $i=1,\ldots,d$).
 For every
 vertex $u\in V(G)\setminus K_d$ suppose that $N_{\sigma}(u)=\{u_{1},u_{2},\ldots,u_{n_{\sigma}(u)}\}$ is well ordered.
 We define a map $\alpha_{\sigma}$  from $E(G)$ to the set of all $2$--subsets of $V(G)$ as follows:
\begin{itemize}
  \item  [(a)] $\alpha_{\sigma}(uv)=uv$ for every edge $uv\in E(K_d)$.
  \item  [(b)] $\alpha_{\sigma}(uu_{i})=ux_i$ for each vertex $u\in V(G)\setminus K_d$ and every $i=1,\ldots,n_{\sigma}(u)$.
\end{itemize}
We define  the graph $\alpha_{\sigma}(G):=(V(G),\alpha_{\sigma}(E(G)))$.
\end{definition}

Figure \ref{chordalthreshold} shows a chordal graph $G$ and the graph $\alpha_{\sigma}(G)=T$. The labeling in the figure corresponds to a PEO $\sigma$ of $G$ that satisfies the conditions of Lemma \ref{K_d}. Vertices $7,8,9,10$ corresponds to $K_d=\{x_1=\sigma(10),x_2=\sigma(9),x_3=\sigma(8),x_4=\sigma(7)\}$. For instance, we observe that if $u\in V(G)\setminus K_d$ is such that $u=\sigma(2)$, as $N_{\sigma}(u)=\{u_{1}=\sigma(3),u_{2}=\sigma(4)\}$ then $ux_1$ and $ux_2$ are edges in $T$. On the other hand, we may observe that the PEO $\sigma$ of $G$ in Figure \ref{chordalthreshold2}  also satisfies the conditions of Lemma \ref{K_d}.

\begin{figure}[htb]
\begin{center}
 \includegraphics[width=.7\textwidth]{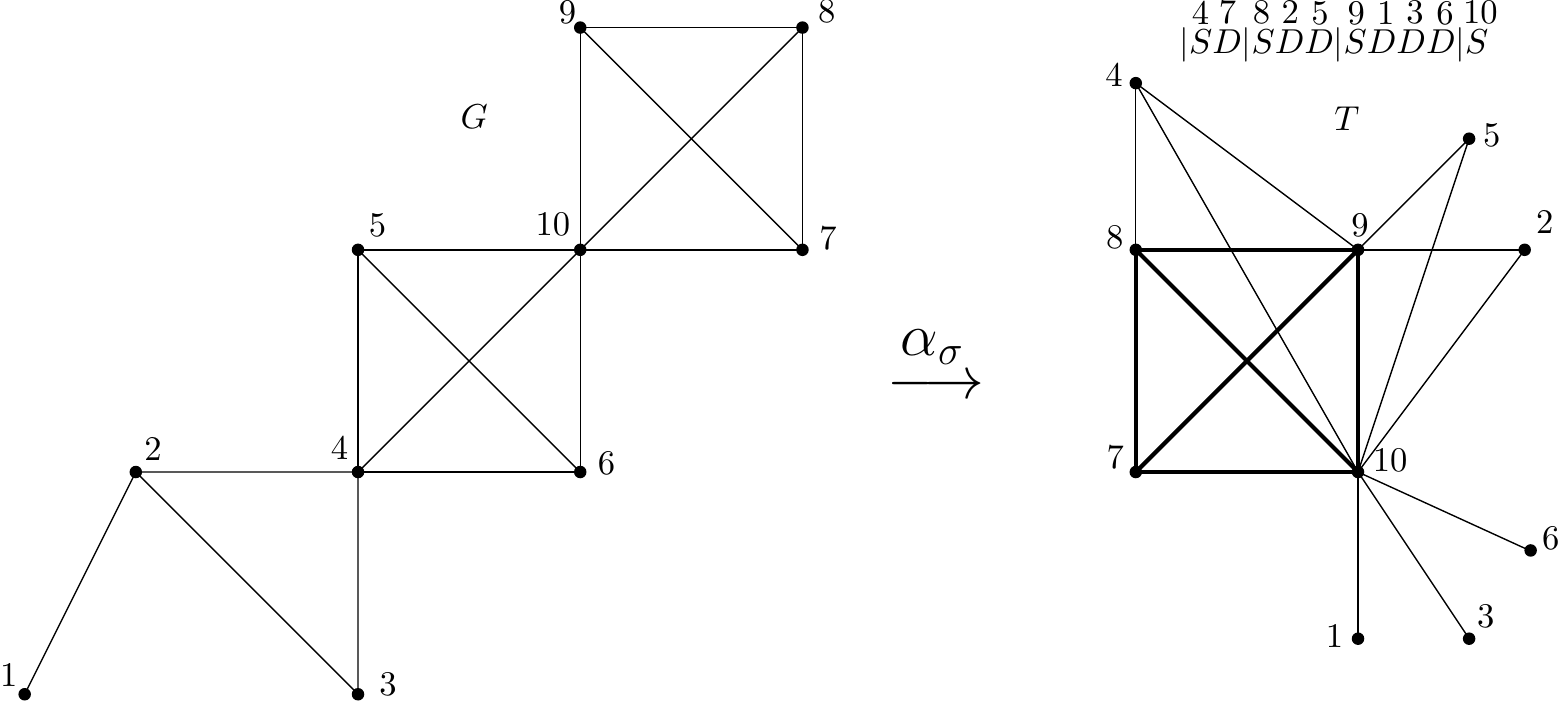}
\caption{ \label{chordalthreshold}}
\end{center}
\end{figure}

We now show that the previous graph shifting $\alpha_{\sigma}(G)$ yields the same $b$--vector of $G$. This is a key tool as it yield the same output as exterior or symmetric shifting. However, its construction allows to see that certain combinatorial properties are preserved.
\begin{theorem}\label{samecliquevector}
Let $G$ be a chordal graph and let $T=\alpha_\sigma(G)$ be as in Definition \ref{combishiftingraph}.
Then  $T$ is a threshold graph and both graphs have the same clique vector and $b$--vector.
\end{theorem}
\begin{proof}

 We note, by the definition of $\alpha_{\sigma}$,  that $T$ is conformed two parts. First, the maximum clique $K_d=\{x_1,x_2\ldots,x_d\}$,  where we may assume that $\sigma^{-1}(x_i)=n-i+1$ for every $i=1,\ldots,d$. Second,  additional vertices $v$ adjacent to $K_d$ such that if $v\in V(T)\setminus K_d$ with $\deg(v)=i$.  
Then, $N(v)=\{x_1,x_2\ldots,x_i\}$.

 In order to prove that $T$ is threshold, it is enough by Theorem \ref{Mahadev} to show that $T$ does not contain graphs $C_4$, $2K_2$ or $P_4$ as an induced subgraph. It is not difficult to see that $T$ does not contain the graphs $C_4$ or $2K_2$ as an induced subgraph.  By means of contradiction, suppose that $T$  contains the graph $P_4=p_1,p_2,p_3,p_4$ as an induced subgraph. By the construction of $T$ we notice that $|P_4\cap K_d|\ge2$. If $|P_4\cap K_d|\ge3$ we would have a triangle in the induced subgraph $P_4$ which is a contradiction. Therefore, $|P_4\cap K_d|=2$, say 
  $P_4\cap K_d=\{ p_2,p_3\}$. Then,  $p_1,p_4\not\in K_d$. With out loss of generality , we assume that $\deg(p_1)\le \deg(p_4)$.
Then, $N(p_1)\subseteq N(p_4)$  by  definition of the graph $T$. We conclude that $p_2p_4\in E(T)$, a contradiction. Similar arguments holds if another two vertices of $P_4$ are in $K_d$. Therefore  the graph $T$ is threshold.

We now show that $G$ and $T$ have the same clique vector, i.e,  $\mathbf{c}(G)=\mathbf{c}(T)$. We note that $|V(G)|=|V(T)|$, and so, $\mathbf{c}(G)$ and $\mathbf{c}(T)$ coincide in its first coordinate.  
We now prove that $\alpha_{\sigma}$ induces a bijection between the set of $r$--cliques of $G$ and the set of $r$--cliques of $T$ for each $r=2,3,\ldots,d$.
 Let $C_u=\{u_1,u_2,\ldots,u_r\}$ be a well ordered $r$--clique of $G$ not contained in $K_d$.
Note that if $u_i\in K_d$ for some $i\in \{1,\ldots,r\}$. Then, $u_{j}\in K_d$ for every $j>i$, because $\sigma^{-1}(u_i)<\sigma^{-1}(u_j)$ and
 $\sigma^{-1}(z)<\sigma^{-1}(x)$ for every $z\in V(G)\setminus K_d$ and every $x\in K_d$. Hence, $u_1\not\in K_d$. Consider $N_{\sigma}(u_1)=\{u_1',u_2',\ldots,u_{n_{\sigma}(u_1)}'\}$, which is well ordered. As
 $\{u_2,u_3\ldots,u_r\}\subseteq N_{\sigma}(u_1)$, there exists $\{j_2,j_3,\ldots,j_r\}\subseteq \{1,\ldots,n_{\sigma}(u_1)\}$ such that $u_{j_q}'=u_q$ for every $q=2,3,\ldots,r$. Therefore, the map $\alpha_{\sigma}$ sends the clique $C_u$ in $G$ to a unique clique $\alpha_{\sigma}(C_u)=\{u_1,x_{j_2},x_{j_3},\ldots,x_{j_r}\}$ in $T$. Let $C_v$ be any $r$--clique  of $G$ such that  $\alpha_{\sigma}(C_u)=\alpha_{\sigma}(C_v)$.  We now show that $C_u=C_v$.
 We observe first that $C_v$ is not contained in $K_d$. 
 Otherwise,  we would have that $\{u_1,x_{j_2},x_{j_3},\ldots,x_{j_r}\}\subseteq K_d$, because the map $\alpha_{\sigma}$ is the identity in $K_d$. This yields a contradiction because $u_1\not\in K_d$. 
 
 Consider the well-ordered sets  $C_v=\{v_1,v_2\ldots,v_r\}$ and  $N_{\sigma}(v_1)=\{v_1',v_2',\ldots,v_{n_{\sigma}(v_1)}'\}$.  Then $v_1\not\in K_d$. Since
  $\{v_2,v_3\ldots,v_r\}\subseteq N_{\sigma}(v_1)$, there exists $\{i_2,i_3,\ldots,i_r\}\subseteq \{1,\ldots,n_{\sigma}(v_1)\}$ such that $v_{i_q}'=v_q$ for every $q=2,3,\ldots,r$. Therefore, the map $\alpha_{\sigma}$ sends the clique $C_v$ in $G$ to a unique clique $\alpha_{\sigma}(C_v)=\{v_1,x_{i_2},x_{i_3},\ldots,x_{i_r}\}$ in $T$.  As $\alpha_{\sigma}(C_u)=\alpha_{\sigma}(C_v)$, it follows that $u_1=v_1$. Since
 $\sigma$ is a bijection $\sigma:\{1,\ldots,n\}\to V(G)$, we conclude that $i_q=j_q$ for every $q=2,3,\ldots,r$. This means $C_u=C_v$. 
 
 Finally,  we have that $\alpha_{\sigma}(C)\neq\alpha_{\sigma}(C')$ for any distinct $i$--cliques $C$ and $C'$, both contained in $K_d$. Therefore $\alpha_{\sigma}$ induces a bijection between the set of $r$--cliques of $G$ and the set of $r$--cliques of $T$ for each $r=2,3,\ldots,d$. We conclude that $\mathbf{c}(G)=\mathbf{c}(T)$. Hence, $\mathbf{b}(G)=\mathbf{b}(T)$.
\end{proof}

\begin{corollary}\label{corocombishift}
Let $G$ be a chordal graph and let $T=\alpha_\sigma(G)$ be as in Definition \ref{combishiftingraph}. Then  $\Delta(T)=\Delta(G)^e$. In particular, $\kappa(G)=\kappa(T)$.
\end{corollary}
\begin{proof}
By Theorems \ref{TeoCarly} and \ref{samecliquevector}, $\Delta(T)$ is shifted and $T$ is threshold respectively. We also have that $\Delta(G)^e$ is shifted. By Theorem \ref{TeoCarly}, there exists a threshold graph $T'$ such that  $\Delta(G)^e=\Delta(T')$. Then, $T'$ and $G$ have the same clique vector.   By Theorem \ref{samecliquevector}, $T$ and $G$ also have the same clique vector.
Since $T$ and $T'$ are threshold graphs with the same clique vector, they also have the same $b$--vector. Then, $T=T'$ and   $\Delta(T)=\Delta(G)^e$.
  Since exterior  shifting operation preserves depth by Theorem \ref{ThmShiftProp}, we
   have that $\Depth(R/I_{\Delta(G)})=\Depth(R/I_{\Delta(T)})$.
   Since $G$ and $T$ are chordal graphs, then $I_{\Delta(G)}$ and $I_{\Delta(T)}$ have $2$-linear resolutions \cite[Theorem 1]{Fro}. We conclude that
   $\kappa(G)=\kappa(T)$ by Corollary \ref{CorDepth}.
\end{proof}

\begin{notation}
Let us define $\widetilde{\kappa}(G):=\max\{|C\cap C'|\mbox{ }\vert \mbox{ }C  \mbox{ and } C' \mbox{ are maximal cliques of $G$ }\},$
\end{notation}

 We note that $\kappa(G)\le \widetilde{\kappa}(G)$.
Next, we  show  that $d_i(T)\le d_i(G)$ and the equality holds for $i> \widetilde{\kappa}(G)$.

\begin{theorem}\label{D_i}
Let $G$ be a chordal graph and let $T=\alpha_\sigma(G)$ as in Definition \ref{combishiftingraph}. Then,
\begin{itemize}
  \item [(a)] $d_i(T)\le d_i(G)$ for every $i=1,\ldots,d$.
    \item [(b)] $d_i(T)=d_i(G)$ for every $i> \widetilde{\kappa}(G)$.
\end{itemize}
\end{theorem}
\begin{proof}By Definition \ref{combishiftingraph},
  $K_d=\{x_1,x_2\ldots,x_d\}$ is a maximum clique of $G$ with $\sigma^{-1}(x_j)=n-j+1$ for every $j=1,\ldots,d$.
\begin{itemize}
\item[(a)] We note that the claim  holds for $i=1$.
 Let $\mathcal{D}_i(G)$ be a minimum dominating $i$--clique in $G$ for some $i\in \{2,3,\ldots,d\}$. 
Let $C_i\in  \mathcal{D}_i(G)$. Suppose that there exits a maximal clique $C$ containing $C_i$ with $i>|s(C)|$ Then, by Lemma \ref{K_d}(c),  there exists a unique vertex,  $u_C$, in $C\setminus s(C)$ such that $n_{\sigma}(u_C)=i-1$.
We define a map $\varphi$ from $\mathcal{D}_i(G)$ to $\mathcal{D}_i(T)$ as follows. For $C_i\in \mathcal{D}_i(G)$, we set
\begin{itemize}
  \item [(i)] If for every maximal clique $C$ of $G$ containing $C_i$ we have $i\le |s(C)|$, we set $\varphi(C_i)=\{x_1,x_2,\ldots,x_i\}$.
    \item [(ii)]  If there exists a maximal clique $C$ containing $C_i$ such that $i>|s(C)|$, we set $\varphi(C_i)=\{x_1,x_2,\ldots,x_{i-1},u_C\}$.
\end{itemize} 
First, we  show that $\varphi$ is a map from $\mathcal{D}_i(G)$ to $\mathcal{D}_i(T)$.
 Suppose that (ii) occurs and that there is a maximal clique $C'$, different from $C$, containing $C_i$ and such that $i>|s(C')|$. By Lemma \ref{K_d}(c),  there exists a unique vertex $u_{C'}\in C'\setminus s(C')$ such that $n_{\sigma}(u_{C'})=i-1$. We prove that $u_C=u_{C'}$. By Lemma \ref{K_d}(d), we may assume without loss of generality that $\sigma^{-1}(v)<\sigma^{-1}(u)$ for every $v\in C\setminus C'$ and every $u\in C\cap C'$. Then, by Lemma \ref{K_d}(b), we have that $s(C)\subset C\cap C'$. If $C=\{u_1,u_2,\ldots,u_{|C|}\}$ is well ordered, it follows that $u_j\in C\cap C'$ for every $j=|C|-|C\cap C'|+1,|C|-|C\cap C'|+2,\ldots,|C|$. Moreover, since $u_j\not\in s(C)$ for every $j=1,\ldots,|C|-|s(C)|$ by Lemma \ref{K_d}(b), it follows that
  $n_{\sigma}(u_j)=|C|-j$ for every $j=1,\ldots,|C|-|s(C)|$. Since $|C\cap C'|\ge i>
|s(C)|$, we have that $|C|-i+1\in \{1,\ldots,|C|-|s(C)|\}$. We conclude that $n_{\sigma}(u_{|C|-i+1})=i-1$. As $u_{|C|-i+1}\in C\cap C'$, it follows that $u_{|C|-i+1}=u_C=u_{C'}$, because $u_C$ and $u_{C'}$ are unique in $C$ and $C'$ respectively with the property   $n_{\sigma}(u_C)=n_{\sigma}(u_{C'})=i-1$. Therefore, $\varphi$ is a map from $\mathcal{D}_i(G)$ to $\mathcal{D}_i(T)$.

Second,  we  show that $\varphi$ is surjective proving that $d_i(T)\le d_i(G)$ for every $i=2,3,\ldots,d$. By Observation \ref{obs}(e), there is only one minimum dominating $i$--clique in $T$ which is formed by $\mathcal{C}_i(T)$, the set of maximal $i$--cliques of $T$, and the $i$--clique $\{x_1,x_2,\ldots,x_i\}$. 
We have that $\{x_1,x_2,\ldots,x_i\}\in \varphi(\mathcal{D}_i(G))$, since there is an $i$--clique $C_i\in \mathcal{D}_i(G)$ such that $C_i\subseteq K_d$. We note that a maximal $i$--clique in $T$ is a clique $\{x_1,x_2,\ldots,x_{i-1},u\}$ where $u\in V(T)\setminus K_d$. We now show that $\{x_1,x_2,\ldots,x_{i-1},u\}\in \varphi(\mathcal{D}_i(G))$.
By the definition of $\alpha_{\sigma}$, we have that $n_{\sigma}(u)=i-1$. We consider a maximal clique $C$ containing
 the $i$--clique $u\cup N_{\sigma}(u)$ in $G$, which means that $u\not\in s(C)$. Thus, it follows   that $s(C)\subset N_{\sigma}(u)$ by Lemma \ref{K_d}(b). As  $|N_{\sigma}(u)|=i-1$, we have that $i>|s(C)|$. We conclude that  $\varphi(u\cup N_{\sigma}(u))=\{x_1,x_2,\ldots,x_{i-1},u\}$, which proves that  $\varphi$ is surjective. Therefore $d_i(T)\le d_i(G)$ for every $i=1,\ldots,d$.

 \item[(b)] Suppose that $\mathcal{D}_i(G)=\{C_1,C_2,\ldots, C_{d_i(G)}\}$.  Let $M_j$ be a maximal clique of $G$ containing $C_j$ for every $j=1,\ldots,d_i(G)$. 
  We have that
 $C_j$ nor $C_r$ are contained in $M_j\cap M_r$ for every distinct $j,r=1,\ldots,d_i(G)$, because $i>\widetilde{\kappa}(G)$.
 As a consequence, $M_j\neq M_r$. By definition of $s(M_j)$, we have that $\widetilde{\kappa}(G)\ge s(M_j)$ concluding that  $i>s(M_j)$. Then, by Lemma \ref{K_d}(c), there exists a unique vertex $u_j\in M_j\setminus s(M_j)$ such that $n_{\sigma}(u_j)=i-1$ for every $j=1,\ldots,d_i(G)$.

We  prove that the map $\varphi$ is injective for $i> \widetilde{\kappa}(G)$. By definition of  $\varphi$, we have that $\varphi(C_j)=\{x_1,x_2,\ldots,x_{i-1},u_j\}$ and $\varphi(C_r)=\{x_1,x_2,\ldots,x_{i-1},u_r\}$ for every distinct $j,r=1,\ldots,d_i(G)$.
We show that $u_j\neq u_r$ proving that $\varphi(C_j)\neq \varphi(C_r)$. We have $u_j\neq u_r$ if $M_j\cap M_r=\emptyset$.
 Suppose $M_j\cap M_r\neq\emptyset$. As $M_j\neq M_r$, we may suppose without loss of generality   that $\sigma^{-1}(u)<\sigma^{-1}(v)$ for every $u\in M_j\setminus M_r$ and every $v\in M_j\cap M_r$ by  Lemma \ref{K_d}(d). Since $i>|M_j\cap M_r|$ and   $n_{\sigma}(u_j)=i-1$,
 it follows that $u_j\in M_j\setminus M_r$. Then, $u_j\neq u_r$, and so, $\varphi$ is injective. Hence, $d_i(G)=d_i(T)$ for every $i> \widetilde{\kappa}(G)$.
 \end{itemize}
\end{proof}

\subsection{$b$--vectors and betti numbers of chordal graphs}

In this subsection we show our main result regarding the $b$--vector of a chordal graph.
In particular, we compute some of the entries of the $b$--vector. We first compute some entries using the combinatorial procedure $\alpha_\sigma$ introduced
in the previous subsection.

\begin{corollary}\label{b_i}
  Let $G$ be a chordal graph, then
  \begin{itemize}
    \item [(a)] $b_i\le d_i(G)$ for every $i=1,\ldots,d$;
    \item [(b)] $b_i=d_i(G)$ for every $i>\widetilde{\kappa}(G)$;
    \item [(c)] $b_i\le b_j$  for every $i,j>\widetilde{\kappa}(G)$ with $j<i$.
  \end{itemize}
\end{corollary}
\begin{proof}
By Theorem \ref{D_i} and Proposition  \ref{dominate}(c), we have (a) and (b). Let $\mathcal{C}_{\ge i}$ be the set of all maximal cliques of $G$ with cardinality at least $i$ and notice that $|\mathcal{C}_{\ge i}|=d_i(G)$ when  $i\ge\widetilde{\kappa}(G)+1$. Hence by (b) we have that $|\mathcal{C}_{\ge i}|=b_i$ for every  $i\ge\widetilde{\kappa}(G)+1$. As  $|\mathcal{C}_{\ge i}|\le|\mathcal{C}_{\ge j}|$ for every $j<i$, we have (c).
\end{proof}
The next example shows that  Theorem \ref{D_i}(b) and Corollary \ref{b_i}(b)and (c) are best possible.
\begin{example}\label{bestpossible}
For every positive integers $\kappa,\widetilde{\kappa}$ with $\kappa\le \widetilde{\kappa}$, consider the graph $G$ conformed by a clique $K=\{x_1,x_2,\ldots,x_{2\widetilde{\kappa}}\}$ of size $2\widetilde{\kappa}$ together with three additional  vertices, $u_{\kappa},u_{\widetilde{\kappa}}$ and $v_{\widetilde{\kappa}}$ such that $N(u_{\kappa})=\{x_1,x_2,\ldots,x_{\kappa}\}$, $N(u_{\widetilde{\kappa}})=\{x_1,x_2,\ldots,x_{\widetilde{\kappa}}\}$ and $N(v_j)=\{x_{\widetilde{\kappa}+1},x_{\widetilde{\kappa}+2},\ldots,x_{2\widetilde{\kappa}}\}$.

We observe that the graph $G$ of Example \ref{bestpossible} is chordal, $\kappa=\kappa(G)$ and $\widetilde{\kappa}=\widetilde{\kappa}(G)$. In addition,
 $b_i=d_i(T)$ for every $i=1,\ldots,d$ where $T$ is the threshold graph such that $\alpha_{\sigma}(G)=T$  by  Proposition \ref{dominate}(c) and Theorem \ref{samecliquevector}. Now, let us consider $i\le \widetilde{\kappa}$. We first notice that $b_i=2<3=d_i(G)$ for  $i=\kappa+1$ (which means $\kappa<\widetilde{\kappa}$ since $i\le \widetilde{\kappa}$). Second, we notice that  $b_i=1<2=d_i(G)$ for $i\neq\kappa+1$. In any case $b_i<d_i(G)$  if $i\le \widetilde{\kappa}$, showing that   Theorem \ref{D_i}(b) and Corollary \ref{b_i}(b) are best possible. Finally, if $i=\kappa=\widetilde{\kappa}$, we have that
$b_i=1<4=b_{i+1}$ showing that  Corollary \ref{b_i}(c) is best possible.
\end{example}

We now study the relation of the $b$--vector of a chordal graph $G$ with  the Betti numbers of the square-free monomial
ideal associated to the clique simplicial complex $\Delta(G)$. We now show that the Betti numbers of $R/I_{\Delta(G)}$ can be expressed by a formula in terms of  its $b$--vector when $G$ is a chordal graph.
\begin{remark}\label{relacion c vs b}
The relation between the $c$--vector and the $b$--vector (given indirectly by  Formula \ref{eq1}) can be written explicitly via the formula $c_i=\sum\limits_{j=i}^{d}\binom{j-1}{j-i}b_j$.
\end{remark}

\begin{proposition}\label{beta y b}
Let $G$ be a chordal graph, 
 then the Betti numbers of $R/I_{\Delta(G)}$ can be expressed by a formula in terms of  its $b$--vector. 
\end{proposition}
\begin{proof}
Since the $f$--vector and the clique vector of $\Delta(G)$ have the relation  $f_{i-1}=c_{i}$ for every $i=1,\ldots,d$ by Remarks \ref{h-vector f-vector} and \ref{relacion c vs b},  the $h$--vector and the $b$--vector of $\Delta(G)$ have the following relation
\begin{equation}\label{h y b}
h_i=\sum\limits_{j=0}^{i}(-1)^{i-j}\binom{d-j}{i-j}\Bigg( \sum\limits_{k=j}^{d}\binom{k-1}{k-j}b_k\Bigg)
\end{equation}
for every $i=0,1,\ldots,d$. Since $G$ is chordal, $I_{\Delta(G)}$ has a $2$-linear resolution \cite[Theorem 1]{Fro}, 
 we have that
 \begin{equation}\label{beta y h}
 \beta_i(R/I_{\Delta(G)})=\sum\limits_{\ell=0}^{2+i}(-1)^{\ell+i+1}h_{2+i-\ell}\binom{n-d}{\ell}
 \end{equation}
 for every $0\le i\le p$  by Theorem \ref{ThmBettiH}. Here, $p$ denotes the projective dimension of the Stanley-Reisner ring $R$. Therefore, by combining Equations (\ref{h y b}) and (\ref{beta y h}), we  obtain a formula of the Betti numbers $\beta_i(R/I_{\Delta(G)})$ in terms of its $b$--vector. Precisely $$\beta_i(R/I_{\Delta(G)})=\sum\limits_{\ell=0}^{2+i}(-1)^{\ell+i+1}\Bigg(\sum\limits_{j=0}^{2+i-\ell}(-1)^{2+i-\ell-j}\binom{d-j}{2+i-\ell-j}
\Bigg(\sum\limits_{k=j}^{d}\binom{k-1}{k-j}b_k\Bigg)
\Bigg)\binom{n-d}{\ell}$$  for every $0\le i\le p$.
\end{proof}

We now extend a theorem of Goodarzi's \cite[Theorem 1]{Goodarzi} regarding $b$--vector and connectivity of $G$. In particular, we describe one more entry in terms of connectivity components.

\begin{theorem}\label{k+1}
  Let $G$ be a chordal graph of order $n$ and vertex connectivity $\kappa$. 
Then, $$b_{i}=
\sum\limits_{|Y|=i-1} (W(G-Y)-1)+1=\beta_{n-i}(R/I_{\Delta(G)})+1$$ for every $i=1,\ldots,\kappa+1$ where $Y\subseteq V(G)$. In addition,
$$b_{i}<
\sum\limits_{|Y|=i-1} (W(G-Y)-1)+1=\beta_{n-i}(R/I_{\Delta(G)})+1$$ for every $i=\kappa+2,\ldots, d$.
\end{theorem}
\begin{proof}
By Proposition \ref{beta}, we have that
$$
\beta_{n-i,n-i+1}(R/I_{\Delta(G)})=\sum_{|Y|=i-1} (W(G-Y)-1)+1=0
$$
for every $i=1,\ldots,\kappa$. Since we already know that $b_1=\ldots=b_\kappa=1$ from Goodarzi's work \cite[Theorem 1]{Goodarzi}, the theorem holds for every $i=1,\ldots,\kappa$.

We now focus for the case $i=\kappa+1.$
Let $T=\alpha_\sigma(G)$ as in Definition \ref{combishiftingraph}, then   $\Delta(T)=\Delta(G)^e$ and $\kappa(T)=\kappa(G)$ by Corollary \ref{corocombishift}.
We have that $b_{\kappa+1}(G) =b_{\kappa+1}(T)$ by Theorem \ref{samecliquevector}. In addition,
 \begin{align*}
\sum_{|Y|=\kappa} (W(G-Y)-1)
&= \beta_{n-i,n-i+1}(R/I_{\Delta(G)}) \hbox{ by Proposition \ref{beta}}\\
&=\beta_{n-i}(R/I_{\Delta(G)}) \hbox{ because }I_{\Delta(G)}\hbox{ has a }2\hbox{-linear resolution.}\\
&=\beta_{n-i}(R/I_{\Delta(T)}) \hbox{ by Proposition  \ref{PropPreservationBetta}}\\
& =\beta_{n-i,n-i+1}(R/I_{\Delta(T)}) \hbox{ because }I_{\Delta(G)}\hbox{ has a }2\hbox{-linear resolution.}\\
&=\sum_{|Y|=\kappa} (W(G-Y)-1) \hbox{ by Proposition \ref{beta}}.
\end{align*}
Then, it suffices to prove our claims for $T$, which were proved in Propositions \ref{dominate}(d) and \ref{k+2} respectively.

\end{proof}

The $b$--vector 
 of $G$ in Figure \ref{chordalthreshold} is $(1,4,3,2)$ and $\kappa(T)=\kappa(G)=\widetilde{\kappa}(T)=\widetilde{\kappa}(G)=1$. We observe that effectively $b_1=1$ and $b_i=d_i(G)$ for every $i\ge 2$ (Corollary \ref{b_i}(b)). Figure \ref{chordalthreshold2} shows another chordal graph $G$ and its graph $\alpha_{\sigma}(G)=T$. We observe that the $b$--vector of $G$ in Figure \ref{chordalthreshold2} is $(1,2,3,2,2)$ and $\kappa(T)=\kappa(G)=1<2=\widetilde{\kappa}(G)<\widetilde{\kappa}(T)=4$. Hence, $b_1=1$ and $b_i=d_i(G)$ for every $i>2=\widetilde{\kappa}(G)$ (Corollary \ref{b_i}(b)). Finally, we observe that edges $(4,10)$ and $(9,10)$ of $T$ are a minimum  dominating $2$--clique in $T$ and edges $(3,5)$, $(4,5)$ and $(6,7)$ are a minimum  dominating $2$--clique in $G$. Hence, $b_2=d_2(T)=2<3=d_2(G)$. Nevertheless, we know that $b_2$ give us another information in  $G$, because  $b_2=\sum\limits_{|Y|=1} (W(G-Y)-1)+1=2$  by Theorem \ref{k+1} .
\begin{figure}[htb]
\begin{center}
 \includegraphics[width=.9\textwidth]{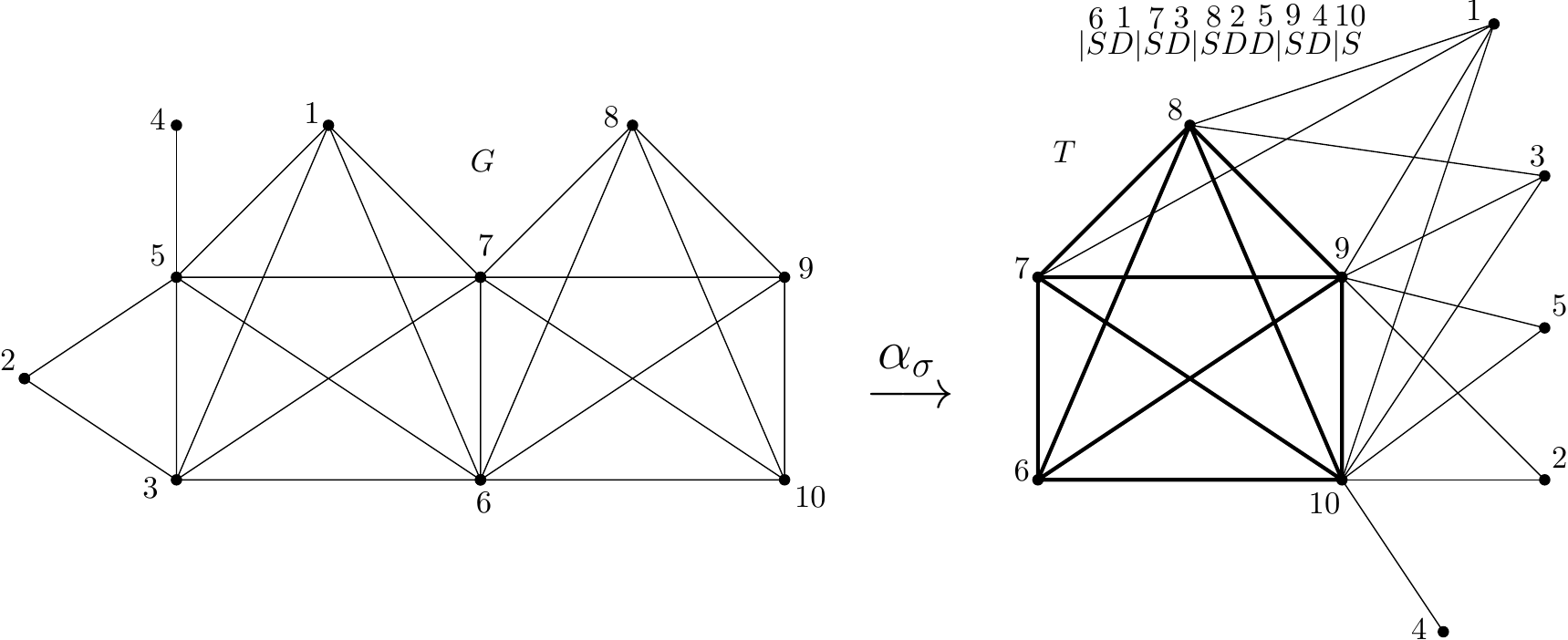}
\caption{ \label{chordalthreshold2}}
\end{center}
\end{figure}

\subsection{Examples}
We finish this section by providing examples of pure and matroid complexes arising from chordal and threshold graphs in which we know the value of its $b$--vector.
A simplicial complex $\Delta$ is \emph{pure} if all its facets have the same cardinality and a \emph{matroid complex} $\Delta_V$ on a set of vertices $V$ is a pure simplicial complex such that $\Delta_{V\setminus S}$ is pure for every $S\subseteq V$.

Klivans \cite{Carly}  proved that given a graph $T$, $\Delta(T)$ is a shifted  matroid complex if and only if $T$ is threshold and its corresponding word is of the form $SDDDSSS$. The first example shows  that the above is true for pure complexes arising from threshold graphs.

\begin{example} \label{examplethresholdpure}
Let $T$ be a threshold graph 
 and suppose that $\Delta (T)$ is pure.
Note that there is no maximal clique of size $i$ in $T$ for $i\le d-1$. Then, by Proposition  \ref{dominate}(i)(ii), $b_i=1$ for $i\le d-1$.  Since $\Delta(T)$ is pure the word corresponding to $T$ is of the form $SDDDSSS$,  $T-S$ is pure for every $S\subset V(G)$. This is because the corresponding word of $T-S$ is has the same type. Hence, $\Delta (T)$ is a matroid complex.
\end{example}

\begin{example} 
Let $G$ be a chordal graph 
 and suppose that  $\Delta(T)$ is a matroid complex.
 We show that $G$ is in fact threshold. For this, it is enough to show that $G$ does not contain the graphs $C_4$, $2K_2$ or $P_4$ as an induced subgraph  by Theorem \ref{Mahadev}.
First, we claim that for every vertex $v\in V(G)$ and every edge $u_1u_2\in E(G)$, we have that $vu_1\in E(G)$ or $vu_2\in E(G)$.
Let $v,u_1,u_2$ different vertices in $G$ such that $u_1u_2\in E(G)$ and let $S=V(G)-\{v,u_1,u_2\}$. As $\Delta (G)$ is a matroid complex, we have that $\Delta (G-S)$ is pure which means that $vu_1\in E(G)$ or $vu_2\in E(G)$.
Second, by the previous claim, $G$ does not contain the graphs $C_4$, $2K_2$ or $P_4$ as an induced subgraph. We obtain that  $G$ is threshold. Finally, we also notice that  $b_i=1$ for $i\le d-1$ by Example \ref{examplethresholdpure}.
\end{example}




The last example is when $G$ is a chordal graph and $\Delta(G)$ is pure. In this case,  the $b$--vector satisfies  $0<b_1\le b_2\le\cdots \le b_d $  \cite[Theorem 1.2]{H08}. Hence, from Corollary \ref{b_i}(c), we obtain the following.

\begin{example}
Let $G$ be a chordal graph with 
 $\Delta(G)$ pure. Then,  $b_{\widetilde{\kappa}(G)+1}=b_{\widetilde{\kappa}(G)+2}=\cdots=b_d$.
\end{example}


\bibliographystyle{alpha}
\bibliography{References}


\end{document}